\documentclass[12pt,reqno]{amsart}

\usepackage{amsmath,amsthm,amssymb,amsfonts}
\usepackage{mathrsfs}
\usepackage{mathtools}
\mathtoolsset{showonlyrefs=true}

\usepackage[a4paper,margin=1.1in]{geometry}
\usepackage{hyperref}
\hypersetup{colorlinks=true,linkcolor=blue,citecolor=blue,urlcolor=blue}

\newtheorem{theorem}{Theorem}[section]

\newtheorem{lemma}[theorem]{Lemma}
\newtheorem{corollary}[theorem]{Corollary}
\theoremstyle{definition}
\newtheorem{definition}[theorem]{Definition}
\newtheorem{remark}[theorem]{Remark}

\numberwithin{equation}{section}

\def\DD{D\kern-.7em\raise0.4ex\hbox{\char '55}\kern.33em}

\title[Sine laws on semigroups with an involutive anti-automorphism]{Sine laws on semigroups with an involutive anti-automorphism: A Levi--Civita approach via left translations}

\author{\DD\d{\u a}ng V\~o Ph\'uc}
\address{Department of Mathematics, FPT University, Quy Nhon AI Campus, An Phu Thinh New Urban Area, Vietnam}
\email{dangphuc150488@gmail.com}
\thanks{ORCID: \url{https://orcid.org/0000-0002-6885-3996}}


\begin{document}

\begin{abstract}
Stetk\ae r's matrix (Levi--Civita) method is a powerful tool for functional equations on semigroups involving a homomorphism $\sigma$, as it yields a finite-dimensional invariant space under right translations and a corresponding matrix formalism. When $\sigma$ is an involutive anti-automorphism, however, the parametrized family of right translations reverses multiplication order in the parameter. In this paper, we resolve this operator-level mismatch by establishing the conjugation identity: letting $J$ denote composition with $\sigma$, we prove
\[
J\,R(\sigma(y))\,J=L(y)\qquad(\forall\,y\in S),
\]
which converts the problematic right translates into left translations. Using this left-translation approach, we obtain an anti-automorphic Levi--Civita closure principle and apply it to the generalized sine law. The classical dichotomy $\beta\in\{\pm1\}$ and the parity relation $f\circ\sigma=\beta f$ are obtained without the bridge hypothesis. Under a natural bridge hypothesis, which is automatically satisfied when there exists a central element $c$ with $f(c)\neq 0$, we obtain the corresponding standard $xy$-addition law and the exact $\sigma$-transformation rule for $g$.
\end{abstract}

\maketitle

\section{Introduction}

The classical trigonometric addition laws (most notably d'Alembert's equation and the sine addition law)
form a cornerstone of the theory of functional equations. Since the pioneering monographs
(e.g.\ \cite{AczelDhombresBook}), a major theme has been to understand how the algebraic structure of the
underlying domain controls the analytic structure of solutions. Over the past decades, this program has been
pushed far beyond abelian groups to non-abelian groups \cite{StetkaerBook} and, more recently, to semigroups,
where the absence of inverses and commutativity forces genuinely new techniques.

A modern semigroup-oriented line of research includes Ebanks's treatment of the sine addition law and
related equations on general semigroups \cite{EbanksAroundSine}, his study of sine subtraction laws involving
either an anti-homomorphic involution or a homomorphic involution \cite{EbanksSineSubtraction}, and
Stetk\ae r's generalized addition/subtraction laws and centrality questions
\cite{StetkaerSineLaw,StetkaerCentrality}. A central technical tool in this literature is the
Levi--Civita method: one shows that a finite-dimensional space generated by solutions is invariant under a
suitable translation action, and then exploits that invariance to derive strong structural constraints.

To motivate the operator viewpoint adopted here, consider the two-parameter family of generalized sine
addition/subtraction laws
\begin{equation}\label{eq:intro-family}
f(x\sigma(y))=f(x)g(y)+\beta\,g(x)f(y)+\gamma\,f(x)f(y),
\qquad \beta\in F^\ast,\ \gamma\in F,
\end{equation}
studied systematically in \cite[Theorem~4.4]{StetkaerSineLaw}. When $\sigma$ is a surjective
homomorphism, Stetk\ae r's right-translation argument yields a corresponding $xy$-addition law and the
transformation rule $f\circ\sigma=\beta f$. If $\sigma$ is additionally involutive, then
$\beta\in\{\pm1\}$ by \cite[Corollary~6.1]{StetkaerSineLaw}.

When $\sigma$ is an \emph{anti-homomorphism}, i.e.\ $\sigma(xy)=\sigma(y)\sigma(x)$, the parametrized family
$y\mapsto R(\sigma(y))$ satisfies the order-reversing composition law
\begin{equation}\label{eq:order-reversal}
R(\sigma(y_1))\,R(\sigma(y_2))
=R(\sigma(y_1)\sigma(y_2))
=R(\sigma(y_2y_1)).
\end{equation}
Thus the family $R(\sigma(\cdot))$ is no longer a right-regular representation in the parameter $y$. Of
course, because $\sigma^2=\mathrm{id}$, one may rewrite \eqref{eq:intro-family} by replacing $y$ with
$\sigma(y)$ and thereby recover a standard $xy$-equation. Our aim here is complementary: we isolate an
intrinsic operator identity explaining how the anti-automorphic family $R(\sigma(\cdot))$ is converted into a
genuine translation action.

\smallskip
\noindent\textbf{Operator-level viewpoint.}
When $\sigma$ is an involutive anti-automorphism, the operators $R(\sigma(y))$, $y\in S$, are order-reversing in
the parameter $y$ by \eqref{eq:order-reversal}. Letting $J$ denote composition with $\sigma$,
$(Jh)(x)=h(\sigma(x))$, we prove the conjugation identity
\begin{equation}\label{eq:intro-conj}
J\,R(\sigma(y))\,J=L(y)\qquad(y\in S),
\end{equation}
which converts the family $R(\sigma(\cdot))$ into left translations. This gives a natural operator-level
reformulation of the anti-automorphic situation and is the viewpoint used throughout the paper.

\smallskip
\noindent\textbf{Main idea: conjugation to left translations.}
Throughout, $\sigma:S\to S$ is assumed to be an involutive anti-automorphism.
The conjugation identity \eqref{eq:intro-conj} (Lemma~\ref{lem:conj}) allows us to transfer the right-translate action
$R(\sigma(\cdot))$ to left translations, which restores the translation-invariance mechanism needed for a
Levi--Civita analysis. In the application to the generalized sine law, the decisive additional input is a
bridge condition of the form
\[
\exists\,d\in S:\qquad f(d)\neq0\quad\text{and}\quad L(d)f\in\mathrm{span}\{f,g\}.
\]
A central element $c\in Z(S)$ with $f(c)\neq0$ is a convenient sufficient condition for this bridge
assumption, but the bridge formulation is the intrinsic one used in the proof.

\smallskip
\noindent\textbf{Main results.}
Our first main result, Theorem~\ref{thm:AH-LC}, is an anti-automorphic Levi--Civita closure theorem for equations of the form
\[
f(x\sigma(y))=f(x)h_1(y)+g(x)h_2(y),
\]
where $\{f,g\}$ is linearly independent and $h_2\not\equiv0$. The theorem shows that $J(V)$ is invariant under left
translations (equivalently, $V$ is invariant under $R(\sigma(\cdot))$), yielding a closure principle tailored
to involutive anti-automorphisms.

Our second main result applies this closure to the generalized sine law \eqref{eq:intro-family}. In
Theorem~\ref{thm:AH44}, under the bridge assumption above, we obtain an analogue of Stetk\ae r's structural
theorem: for linearly independent solutions $(f,g)$, one necessarily has $\beta\in\{\pm1\}$, the equation yields
the standard $xy$-addition law, and the expected $\sigma$-transformation rules
\[
f\circ\sigma=\beta f,\qquad g\circ\sigma=
\begin{cases}
g+af,& \beta=-1,\\
g,& \beta=1.
\end{cases}
\]
Moreover, the proof shows that the parity relation $f\circ\sigma=\beta f$ is obtained before the bridge
assumption is invoked, whereas the bridge assumption is used only to obtain control of $g\circ\sigma$ and thereby to
reach the standard $xy$-addition law. Corollary~\ref{cor:AH44-central} records the frequently useful
central-element criterion, and Example~\ref{ex:s3-bridge} shows that the bridge assumption is strictly weaker.

\smallskip
\noindent\textbf{Organization of the paper.}
Section~\ref{sec:prelim} collects the needed definitions and notation (central/abelian functions, parity, and
the regular actions $R$ and $L$). In Section~\ref{sec:AHLC} we develop the conjugation-based framework: we prove
\eqref{eq:intro-conj} and establish the anti-automorphic Levi--Civita closure theorem (Theorem~\ref{thm:AH-LC}).
Section~\ref{sec:AH44} applies this closure to the generalized sine law and proves our main structural
classification under the bridge assumption; Corollary~\ref{cor:AH44-central} gives the corresponding result for central elements. 
Section~\ref{sec:examples} provides concrete illustrations from matrix groups and the symmetric
group, including a non-central example showing that the bridge assumption can hold even when no central element
$c$ satisfies $f(c)\neq0$. Finally, Section~\ref{sec:conclusions} gives concluding remarks and connections to
recent centrality/abelianity results in the literature.

\section{Preliminaries}\label{sec:prelim}
Throughout, $S$ is a semigroup with multiplicative notation; $F$ is a field with $\mathrm{char}(F)\neq2$; $F^\ast=F\setminus\{0\}$. We write $\mathscr{F}(S, F)$ for the space of $F$-valued functions on~$S$.

\begin{definition}[Central and abelian functions]\label{def:central-abelian}
Following \cite[Definition~1.1]{StetkaerCentrality}, let $A$ be a set. A function $H:S\to A$ is \emph{central} if $H(xy)=H(yx)$ for all $x,y\in S$, and it is
\emph{abelian} if
\[
H(x_1x_2\cdots x_n)=H(x_{\pi(1)}x_{\pi(2)}\cdots x_{\pi(n)})
\]
for every integer $n\ge2$, all $x_1,\dots,x_n\in S$, and every permutation $\pi$ of
$\{1,\dots,n\}$. A central function is abelian if and only if it satisfies Kannappan's condition
$H(xyz)=H(xzy)$ for all $x,y,z\in S$.
\end{definition}

\begin{definition}[Parity \cite{StetkaerSineLaw, StetkaerCentrality}]\label{def:parity}
Given a map $\varphi:S\to S$, a function $H:S\to F$ is called \emph{even with respect to} $\varphi$ if $H\circ\varphi=H$, and \emph{odd with respect to} $\varphi$ if $H\circ\varphi=-H$.
\end{definition}

\begin{definition}[Homomorphisms, anti-homomorphisms]\label{def:hom-anti}
A map $\varphi:S\to S$ is a \emph{homomorphism} if $\varphi(xy)=\varphi(x)\varphi(y)$, and an \emph{anti-homomorphism} if $\varphi(xy)=\varphi(y)\varphi(x)$. An \emph{involutive anti-automorphism} is a bijective anti-homomorphism $\sigma$ with $\sigma^2=\mathrm{id}$.
\end{definition}

\begin{definition}[Right/left regular representations]\label{def:R-L}
The operators $R,L: S\to \mathrm{End}\big(\mathscr{F}(S, F)\big)$ are defined by
\[
(R(y)h)(x):=h(xy),\qquad (L(y)h)(x):=h(yx).
\]
Then $R$ is a (semi)group homomorphism and $L$ is an anti-homomorphism:
\[
R(y_1)R(y_2)=R(y_1y_2),\qquad L(y_2)L(y_1)=L(y_1y_2).
\]
\end{definition}

\section{A Left-Translation Method via Operator Conjugation}\label{sec:AHLC}
Throughout this section, $\sigma:S\to S$ is an involutive anti-automorphism.

\subsection{Conjugating Right Translates into Left Translates}
Define an involutive linear operator $J:\mathscr{F}(S, F)\to \mathscr{F}(S, F)$ by
\[
(Jh)(x):=h(\sigma(x)).
\]
Since $\sigma$ is involutive, $J^2=\mathrm{id}$ and $J^{-1}=J$.

\begin{lemma}[Conjugation Identity]\label{lem:conj}
For all $y\in S$, the operator identity holds:
\[
J\,R(\sigma(y))\,J = L(y).
\]
\end{lemma}
\begin{proof}
Fix $h\in \mathscr{F}(S, F)$ and $x\in S$. Using the definitions of $J$ and $R$, and the fact that
$\sigma$ is an anti-homomorphism and involutive, we compute
\begin{align*}
(JR(\sigma(y))Jh)(x)
&=(R(\sigma(y))Jh)(\sigma(x))\\
&=(Jh)(\sigma(x)\sigma(y))\\
&=h\!\left(\sigma(\sigma(x)\sigma(y))\right)\\
&=h\!\left(\sigma(\sigma(yx))\right) \qquad(\text{since }\sigma(x)\sigma(y)=\sigma(yx))\\
&=h(yx)\qquad(\text{since }\sigma^2=\mathrm{id})\\
&=(L(y)h)(x).
\end{align*}
\end{proof}

\subsection{Anti-automorphism Levi--Civita closure}\label{subsec:AHLC}

\begin{theorem}\label{thm:AH-LC}
Let $\sigma:S\to S$ be an involutive anti-automorphism.
Suppose $f,g,h_1,h_2:S\to F$ satisfy
\begin{equation}\label{eq:LC-sigma}
f(x\sigma(y))=f(x)h_1(y)+g(x)h_2(y),\quad x,y\in S,
\end{equation}
where $\{f,g\}$ is linearly independent and $h_2\not\equiv 0$.
Set $V:=\mathrm{span}\{f,g\}$ and $V^\sigma:=J(V)=\mathrm{span}\{Jf,Jg\}$. Then $V^\sigma$ is invariant under $L(y)$ for every $y\in S$.
Equivalently, $V$ is invariant under $R(\sigma(y))$ for every $y\in S$.
\end{theorem}

\begin{proof}

Write \eqref{eq:LC-sigma} in operator form:
\[
R(\sigma(y))f = h_1(y)f + h_2(y)g.
\]
Conjugate by $J$ and use Lemma~\ref{lem:conj}:
\[
L(y)(Jf)=J R(\sigma(y))J(Jf)=J(R(\sigma(y))f)=h_1(y)Jf+h_2(y)Jg.
\]
Thus $L(y)(Jf)\in V^\sigma$.

\medskip
We now show that $L(y)(Jg)\in V^\sigma$ for all $y\in S$. Indeed, choose $y_0\in S$ such that $h_2(y_0)\neq 0$ (this exists because $h_2\not\equiv 0$).
From
\[
L(y_0)(Jf)=h_1(y_0)Jf+h_2(y_0)Jg
\]
we obtain
\[
Jg=\frac{1}{h_2(y_0)}\,L(y_0)(Jf)-\frac{h_1(y_0)}{h_2(y_0)}\,Jf.
\]
Now fix $y\in S$ and apply $L(y)$:
\[
L(y)(Jg)=\frac{1}{h_2(y_0)}\,L(y)L(y_0)(Jf)-\frac{h_1(y_0)}{h_2(y_0)}\,L(y)(Jf).
\]
Since $L$ is an anti-homomorphism (Definition~\ref{def:R-L}), $L(y)L(y_0)=L(y_0y)$. Hence
\[
L(y)(Jg)=\frac{1}{h_2(y_0)}\,L(y_0y)(Jf)-\frac{h_1(y_0)}{h_2(y_0)}\,L(y)(Jf).
\]
We already know $L(z)(Jf)\in V^\sigma$ for all $z\in S$, so the right-hand side lies in $V^\sigma$.
This proves the claim, hence $V^\sigma$ is $L$-invariant.
\end{proof}

\begin{remark}[Sharpness of the hypothesis $h_2\not\equiv 0$]\label{rmk:h2-sharp}
Theorem~\ref{thm:AH-LC} does not require linear independence of $h_1$ and $h_2$; in particular, it applies when
$h_1=0$ and $h_2\not\equiv 0$. The proof uses only the existence of a point $y_0\in S$ with $h_2(y_0)\neq 0$.

This hypothesis is sharp. If $h_2\equiv 0$, the conclusion can fail. For example, let $S=C_3=\{e,a,a^2\}$,
let $\sigma=\mathrm{id}_S$, let $f\equiv 1$, $h_1\equiv 1$, $h_2\equiv 0$, and let $g=\mathbf{1}_{\{e\}}$.
Then
\[
f(x\sigma(y))=1=f(x)h_1(y)+g(x)h_2(y)\qquad(\forall x,y\in S),
\]
and $f,g$ are linearly independent. However,
\[
L(a)g=\mathbf{1}_{\{a^2\}}\notin \mathrm{span}\{1,\mathbf{1}_{\{e\}}\}=\mathrm{span}\{f,g\}.
\]
Hence $V^\sigma=V$ is not invariant under left translations.
\end{remark}

\section{The Generalized Sine Law with an Anti-Automorphism}\label{sec:AH44}

In this section we study the generalized sine law
\begin{equation}\label{eq:gen-anti}
f(x\sigma(y))=f(x)g(y)+\beta\,g(x)f(y)+\gamma\,f(x)f(y),
\qquad x,y\in S,
\end{equation}
where $\beta\in F^\ast$, $\gamma\in F$, and $f,g$ are linearly independent.

\begin{lemma}[A coefficient dichotomy]\label{lem:b01}
Let $S$ be a semigroup and $F$ a field with $\mathrm{char}(F)\neq2$.
Assume that $f,g:S\to F$ are linearly independent and that
\[
f(xy)=f(x)g(y)+b\,g(x)f(y)+c\,f(x)f(y)\qquad(\forall x,y\in S),
\]
where $b,c\in F$ are constants. Then $b\in\{0,1\}$.
\end{lemma}

\begin{proof}
See \cite[Lemma~5.1]{StetkaerCentrality}; cf.\ also \cite[Lemma~3.3]{StetkaerSineLaw}.
\end{proof}

\begin{lemma}\label{cor:AHLC-closure}
Assume \eqref{eq:LC-sigma}. Define
\[
\widehat h_1:=h_1\circ\sigma,\qquad \widehat h_2:=h_2\circ\sigma .
\]
Then
\[
f(xy)=f(x)\widehat h_1(y)+g(x)\widehat h_2(y)\qquad(\forall\,x,y\in S).
\]
\end{lemma}

\begin{proof}
Replace $y$ by $\sigma(y)$ in \eqref{eq:LC-sigma} and use $\sigma^2=\mathrm{id}$:
\[
f(xy)=f\bigl(x\sigma(\sigma(y))\bigr)=f(x)h_1(\sigma(y))+g(x)h_2(\sigma(y))
=f(x)\widehat h_1(y)+g(x)\widehat h_2(y).
\]
\end{proof}

\begin{lemma}[Bridge criterion for $\sigma$-conjugate span]\label{lem:bridge-span}
Let $S$ be a semigroup and let $\sigma:S\to S$ be an involutive anti-automorphism.
Assume that $f,u,h:S\to F$ satisfy
\begin{equation}\label{eq:LC-huf}
f(x\sigma(y))=f(x)\,h(y)+u(x)\,f(y)\qquad(\forall x,y\in S),
\end{equation}
where $f$ and $u$ are linearly independent. Assume moreover that
\begin{equation}\label{eq:parity-f}
f\circ\sigma=\varepsilon f\qquad\text{for some }\varepsilon\in\{\pm1\}.
\end{equation}
If there exists $d\in S$ such that
\begin{equation}\label{eq:bridge-span-assumption}
f(d)\neq 0
\qquad\text{and}\qquad
L(d)f\in \mathrm{span}\{f,u\},
\end{equation}
then
\[
u\circ\sigma\in \mathrm{span}\{f,u\}.
\]
\end{lemma}

\begin{proof}
Fix $y\in S$ and $x\in S$. Since $\sigma$ is an involutive anti-automorphism,
\[
\sigma(yx)=\sigma(x)\sigma(y)\quad\Rightarrow\quad yx=\sigma(\sigma(x)\sigma(y)).
\]
Using \eqref{eq:parity-f}, we have
\[
(L(y)f)(x)=f(yx)=f\!\left(\sigma(\sigma(x)\sigma(y))\right)
=\varepsilon\, f(\sigma(x)\sigma(y)).
\]
Now apply \eqref{eq:LC-huf} to $(x,y)=(\sigma(x),y)$:
\[
f(\sigma(x)\sigma(y))=f(\sigma(x))h(y)+u(\sigma(x))f(y).
\]
Using $f(\sigma(x))=\varepsilon f(x)$ and $u(\sigma(x))=(u\circ\sigma)(x)$, we obtain
\[
f(\sigma(x)\sigma(y))=\varepsilon f(x)h(y)+(u\circ\sigma)(x)\,f(y).
\]
Multiplying by $\varepsilon$ and using $\varepsilon^2=1$, we get
\begin{equation}\label{eq:Lyf-formula}
L(y)f = h(y)\,f+\varepsilon f(y)\,(u\circ\sigma)\qquad(\forall y\in S).
\end{equation}

Set $V:=\mathrm{span}\{f,u\}$. Taking $y=d$ in \eqref{eq:Lyf-formula} and using
\eqref{eq:bridge-span-assumption}, we obtain
\[
L(d)f = h(d)\,f+\varepsilon f(d)\,(u\circ\sigma)\in V.
\]
Since $f(d)\neq0$, we can solve for $u\circ\sigma$:
\[
u\circ\sigma
=\frac{1}{\varepsilon f(d)}\bigl(L(d)f-h(d)f\bigr)\in V.
\]
Thus $u\circ\sigma\in \mathrm{span}\{f,u\}$.
\end{proof}

\begin{corollary}[Central-element criterion]\label{cor:central-span}
Let $S$ be a semigroup and let $\sigma:S\to S$ be an involutive anti-automorphism.
Assume that $f,u,h:S\to F$ satisfy \eqref{eq:LC-huf}, where $f$ and $u$ are linearly independent, and that
\eqref{eq:parity-f} holds for some $\varepsilon\in\{\pm1\}$.
If there exists a central element $c\in Z(S)$ such that $f(c)\neq 0$, then
\[
u\circ\sigma\in \mathrm{span}\{f,u\}.
\]

\smallskip
In particular, the conclusion holds whenever $S$ is commutative and $f\not\equiv 0$.
\end{corollary}

\begin{proof}
Set $V:=\mathrm{span}\{f,u\}$ and let $y_0:=\sigma(c)$. Then $\sigma(y_0)=c$, so by \eqref{eq:LC-huf},
\[
R(c)f=R(\sigma(y_0))f=h(y_0)\,f+f(y_0)\,u\in V.
\]
Since $c\in Z(S)$, we have $L(c)f=R(c)f\in V$. Thus \eqref{eq:bridge-span-assumption} holds with $d=c$, and Lemma~\ref{lem:bridge-span} gives the desired conclusion.

For the commutative case, choose $d\in S$ with $f(d)\neq0$.
Since $\sigma$ is bijective, $d=\sigma(\sigma(d))$, and \eqref{eq:LC-huf} with $y=\sigma(d)$ gives
\[
R(d)f=R(\sigma(\sigma(d)))f\in V.
\]
Because $S$ is commutative, $L(d)f=R(d)f\in V$, so \eqref{eq:bridge-span-assumption} holds and Lemma~\ref{lem:bridge-span} applies.
\end{proof}

\begin{remark}\label{rmk:linear-independence}
Strictly speaking, the algebraic derivations in Theorem~\ref{thm:AH-LC}, Lemma~\ref{lem:bridge-span}, and Corollary~\ref{cor:central-span} do not rely on the linear independence of the functions involved; their conclusions remain valid even if the function pairs are linearly dependent. We state them with the linear independence hypothesis primarily for contextual consistency, as it seamlessly sets up the non-degenerate framework required for our main structural classification in the following important theorem.
\end{remark}

\begin{theorem}\label{thm:AH44}
Let $\sigma:S\to S$ be an involutive anti-automorphism.
Suppose $f,g:S\to F$ are linearly independent and satisfy \eqref{eq:gen-anti}.
Assume that there exists $d\in S$ such that
\begin{equation}\label{eq:AH44-bridge}
f(d)\neq0
\qquad\text{and}\qquad
L(d)f\in \mathrm{span}\{f,g\}.
\end{equation}

\begin{enumerate}
\item[(i)] If $\beta=-1$, then $\gamma=0$. Furthermore, there exists $a\in F$ such that
\[
f(xy)=f(x)g(y)+g(x)f(y)+a\,f(x)f(y),
\]
and
\[
f\circ\sigma=-f,\qquad g\circ\sigma=g+af.
\]
\item[(ii)] If $\beta\neq -1$, then necessarily $\beta=1$, and
\[
f(xy)=f(x)g(y)+g(x)f(y)+\gamma\,f(x)f(y),
\]
and
\[
f\circ\sigma=f,\qquad g\circ\sigma=g.
\]
\end{enumerate}
\end{theorem}

\begin{proof}
We distinguish the two cases $\beta=-1$ and $\beta\neq-1$.

\medskip
\noindent\textbf{(i) The case $\beta=-1$.}
Then \eqref{eq:gen-anti} becomes
\begin{equation}\label{eq:beta=-1-again}
f(x\sigma(y))=f(x)\bigl(g(y)+\gamma f(y)\bigr)-g(x)f(y)\qquad(x,y\in S).
\end{equation}
Set
\[
h:=g+\gamma f,\qquad u:=-g.
\]
Then \eqref{eq:beta=-1-again} can be written in the Levi--Civita form \eqref{eq:LC-sigma} as
\begin{equation}\label{eq:beta=-1-LCform}
f(x\sigma(y))=f(x)\,h(y)+u(x)\,f(y)\qquad(x,y\in S),
\end{equation}
i.e.\ \eqref{eq:LC-sigma} holds with $(f,g,h_1,h_2)=(f,u,h,f)$.
By Lemma~\ref{cor:AHLC-closure} we therefore have, for all $x,y\in S$,
\begin{equation}\label{eq:beta=-1-xy}
f(xy)=f(x)h(\sigma(y))+u(x)f(\sigma(y))
=f(x)\bigl(g(\sigma(y))+\gamma f(\sigma(y))\bigr)-g(x)f(\sigma(y)).
\end{equation}

\smallskip
\noindent\emph{Step 1: a two--variable law and parity of $f$.}
Apply Theorem~\ref{thm:AH-LC} to \eqref{eq:beta=-1-LCform}. Then
\[
V:={\rm span}\{f,u\}={\rm span}\{f,g\}
\]
is invariant under $R(\sigma(z))$ for every $z\in S$.
Hence, for each $z\in S$, there exist functions $P,Q:S\to F$ such that
\begin{equation}\label{eq:g-decomp-i}
g(x\sigma(z))=P(z)\,f(x)+Q(z)\,g(x)\qquad(\forall x\in S).
\end{equation}

Fix $x,y,z\in S$. Since $\sigma$ is an anti-automorphism, $\sigma(yz)=\sigma(z)\sigma(y)$, so
$f(x\sigma(yz))=f(x\sigma(z)\sigma(y))$.
Compute both sides using \eqref{eq:beta=-1-again}.

\smallskip\noindent
(1) With $yz$:
\[
f(x\sigma(yz))=f(x)\bigl(g(yz)+\gamma f(yz)\bigr)-g(x)f(yz).
\]

\smallskip\noindent
(2) With $x$ replaced by $x\sigma(z)$:
\[
f(x\sigma(z)\sigma(y))
=f(x\sigma(z))\bigl(g(y)+\gamma f(y)\bigr)-g(x\sigma(z))f(y).
\]
Insert \eqref{eq:beta=-1-again} for $f(x\sigma(z))$ and \eqref{eq:g-decomp-i} for $g(x\sigma(z))$, then compare
coefficients of the linearly independent functions $f(x)$ and $g(x)$. This yields, for all $y,z\in S$,
\begin{equation}\label{eq:casei-yzlaw}
f(yz)=f(z)g(y)+Q(z)f(y)+\gamma f(z)f(y).
\end{equation}
Putting $y=x$ gives
\begin{equation}\label{eq:casei-fxz}
f(xz)=f(z)g(x)+Q(z)f(x)+\gamma f(z)f(x)\qquad(\forall x,z\in S).
\end{equation}

On the other hand, taking $y=\sigma(z)$ in \eqref{eq:beta=-1-again} gives
\begin{equation}\label{eq:casei-fxz2}
f(xz)=f(x)g(\sigma(z)) - g(x)f(\sigma(z)) + \gamma f(x)f(\sigma(z))\qquad(\forall x,z\in S).
\end{equation}
Comparing \eqref{eq:casei-fxz} and \eqref{eq:casei-fxz2} and using linear independence of $f$ and $g$ yields,
for all $z\in S$,
\begin{equation}\label{eq:casei-parity}
f(\sigma(z))=-f(z),\qquad g(\sigma(z))=Q(z)+2\gamma f(z).
\end{equation}
In particular, $f\circ\sigma=-f$.

\medskip
\noindent\emph{Step 2: obtain $g\circ\sigma\in\mathrm{span}\{f,g\}$.}
Rewrite \eqref{eq:beta=-1-LCform} as
\[
f(x\sigma(y))=f(x)\,h(y)+u(x)\,f(y),\qquad h=g+\gamma f,\ \ u=-g.
\]
From Step~1 we already have $f\circ\sigma=-f$, i.e.\ \eqref{eq:parity-f} holds with $\varepsilon=-1$.
Since $u=-g$, the bridge hypothesis \eqref{eq:AH44-bridge} reads
\[
f(d)\neq0
\qquad\text{and}\qquad
L(d)f\in\mathrm{span}\{f,u\}.
\]
Applying Lemma~\ref{lem:bridge-span} to \eqref{eq:LC-huf}, we conclude that
\[
u\circ\sigma\in \mathrm{span}\{f,u\}.
\]
Equivalently,
\begin{equation}\label{eq:gsigma-in-span-casei}
g\circ\sigma\in \mathrm{span}\{f,g\}.
\end{equation}

\medskip
\noindent\emph{Step 3: exclude $g\circ\sigma=-g+\mu f$.}
By \eqref{eq:gsigma-in-span-casei} there exist $\lambda,\mu\in F$ such that
\[
g\circ\sigma=\lambda g+\mu f.
\]
Apply $\sigma$ again and use $f\circ\sigma=-f$:
\[
g=\lambda(g\circ\sigma)+\mu(f\circ\sigma)
=\lambda(\lambda g+\mu f)-\mu f
=\lambda^2 g+(\lambda\mu-\mu)f.
\]
By linear independence of $f$ and $g$, we obtain
\[
\lambda^2=1,\qquad (\lambda-1)\mu=0.
\]
Assume $\lambda=-1$. Then $g\circ\sigma=-g+\mu f$.
Replacing $y$ by $\sigma(y)$ in \eqref{eq:beta=-1-again} and using $f\circ\sigma=-f$ gives
\[
f(xy)=f(x)\bigl(g(\sigma(y))+\gamma f(\sigma(y))\bigr)-g(x)f(\sigma(y))
=f(x)\bigl(-g(y)+(\mu-\gamma)f(y)\bigr)+g(x)f(y).
\]
Let $\widehat g:=-g$. Then
\[
f(xy)=f(x)\widehat g(y)-\widehat g(x)f(y)+(\mu-\gamma)\,f(x)f(y),
\]
which contradicts Lemma~\ref{lem:b01}. Hence $\lambda\neq -1$, so $\lambda=1$ and
\begin{equation}\label{eq:casei-gsigma-closed}
g\circ\sigma=g+a f\quad\text{for some }a\in F.
\end{equation}

\medskip
\noindent\emph{Step 4: derive $\gamma=0$ and finish.}
From \eqref{eq:beta=-1-xy}, using $f\circ\sigma=-f$ and \eqref{eq:casei-gsigma-closed}, we obtain for all $x,y\in S$,
\begin{align*}
f(xy)
&=f(x)\bigl(g(\sigma(y))+\gamma f(\sigma(y))\bigr)-g(x)f(\sigma(y))\\
&=f(x)\bigl(g(y)+(a-\gamma)f(y)\bigr)+g(x)f(y).
\end{align*}
Thus
\begin{equation}\label{eq:casei-xy-law}
f(xy)=f(x)g(y)+g(x)f(y)+(a-\gamma)f(x)f(y)\qquad(\forall x,y\in S).
\end{equation}

Now compute $-f(xy)$ in two ways. From \eqref{eq:casei-xy-law},
\[
-f(xy)=-f(x)g(y)-g(x)f(y)-(a-\gamma)f(x)f(y).
\]
On the other hand, since $f\circ\sigma=-f$ and $\sigma$ is an anti-automorphism,
\[
-f(xy)=(f\circ\sigma)(xy)=f(\sigma(xy))=f(\sigma(y)\sigma(x)).
\]
Apply \eqref{eq:beta=-1-again} with $(x,y)=(\sigma(y),x)$ and use $f\circ\sigma=-f$ and \eqref{eq:casei-gsigma-closed}:
\begin{align*}
-f(xy)
&=f(\sigma(y)\sigma(x))
=f(\sigma(y))\bigl(g(x)+\gamma f(x)\bigr)-g(\sigma(y))f(x)\\
&=-f(y)g(x)-g(y)f(x)-(a+\gamma)f(x)f(y).
\end{align*}
Comparing the two expressions for $-f(xy)$, we obtain
\[
2\gamma\,f(x)f(y)=0\qquad(\forall x,y\in S).
\]
Since $f\not\equiv 0$, choose $x_0\in S$ with $f(x_0)\neq 0$; taking $x=y=x_0$ and using $\operatorname{char}(F)\neq 2$, we get $\gamma=0$.

With $\gamma=0$, \eqref{eq:casei-xy-law} becomes
\[
f(xy)=f(x)g(y)+g(x)f(y)+a\,f(x)f(y),
\]
and we already have $f\circ\sigma=-f$ and $g\circ\sigma=g+a f$. This proves \textup{(i)}.

\medskip
\noindent\textbf{(ii) The case $\beta\neq-1$.}
Define
\[
g_1:=g+\frac{\gamma}{1+\beta}f.
\]
Then $f$ and $g_1$ are linearly independent and \eqref{eq:gen-anti} is equivalent to
\begin{equation}\label{eq:beta-neq-1-rewrite}
f(x\sigma(y))=f(x)g_1(y)+\beta\,g_1(x)f(y)\qquad(x,y\in S).
\end{equation}

Apply Theorem~\ref{thm:AH-LC} to \eqref{eq:beta-neq-1-rewrite}. Then
$V:={\rm span}\{f,g_1\}$ is invariant under $R(\sigma(z))$ for every $z\in S$.
Hence, for each $z\in S$, there exist functions $P,Q:S\to F$ such that
\begin{equation}\label{eq:g1-decomp}
g_1(x\sigma(z))=P(z)\,f(x)+Q(z)\,g_1(x)\qquad(\forall x\in S).
\end{equation}

Using associativity as before,
$f(x\sigma(yz))=f(x\sigma(z)\sigma(y))$.
Computing both sides via \eqref{eq:beta-neq-1-rewrite} and using \eqref{eq:g1-decomp}, and then comparing coefficients of
the linearly independent functions $f(x)$ and $g_1(x)$, we obtain for all $y,z\in S$,
\begin{align}
g_1(yz) &= g_1(z)g_1(y)+\beta\,P(z)f(y), \label{eq:caseii-g1law}\\
f(yz)  &= f(z)g_1(y)+Q(z)f(y). \label{eq:caseii-fyz}
\end{align}

Putting $y=x$ in \eqref{eq:caseii-fyz} gives
\begin{equation}\label{eq:caseii-fxz}
f(xz)=f(z)\,g_1(x)+Q(z)\,f(x)\qquad(\forall x,z\in S).
\end{equation}
Taking $y=\sigma(z)$ in \eqref{eq:beta-neq-1-rewrite} gives
\begin{equation}\label{eq:caseii-fxz2}
f(xz)=f(x)\,g_1(\sigma(z))+\beta\, g_1(x)f(\sigma(z))\qquad(\forall x,z\in S).
\end{equation}
Comparing \eqref{eq:caseii-fxz} and \eqref{eq:caseii-fxz2} yields, for all $z\in S$,
\begin{equation}\label{eq:caseii-parity}
Q(z)=g_1(\sigma(z)),\qquad f(z)=\beta\, f(\sigma(z)).
\end{equation}
Replacing $z$ by $\sigma(z)$ in $f(z)=\beta f(\sigma(z))$ and using
$\sigma^2={\rm id}$, we obtain
\[
f(\sigma(z))=\beta f(z)\qquad(\forall z\in S).
\]
Hence
\[
f(z)=\beta f(\sigma(z))=\beta^2 f(z)\qquad(\forall z\in S).
\]
Thus $(1-\beta^2)f(z)=0$ for all $z\in S$. Since $f\not\equiv 0$, choose $z_0\in S$ with $f(z_0)\neq 0$. Then $(1-\beta^2)f(z_0)=0$, hence $\beta^2=1$.
Since $\beta\neq-1$, we conclude
\begin{equation}\label{eq:caseii-beta1}
\beta=1\quad\text{and hence}\quad f\circ\sigma=f.
\end{equation}

\medskip
\noindent\emph{Step 2: show $g_1\circ\sigma=g_1$.}
From Step~1, we already have $\beta=1$ and $f\circ\sigma=f$.
With $\beta=1$, equation \eqref{eq:beta-neq-1-rewrite} reads
\[
f(x\sigma(y))=f(x)g_1(y)+g_1(x)f(y)\qquad(\forall x,y\in S).
\]
This is exactly \eqref{eq:LC-huf} with $u=h=g_1$ and $\varepsilon=1$ in \eqref{eq:parity-f}.
Since $\mathrm{span}\{f,g_1\}=\mathrm{span}\{f,g\}$, the bridge hypothesis \eqref{eq:AH44-bridge} gives
\[
f(d)\neq0
\qquad\text{and}\qquad
L(d)f\in\mathrm{span}\{f,g_1\}.
\]
Applying Lemma~\ref{lem:bridge-span}, we get
\[
g_1\circ\sigma\in\mathrm{span}\{f,g_1\}.
\]
Thus $g_1\circ\sigma=\lambda g_1+\mu f$ for some $\lambda,\mu\in F$.
Applying $\sigma$ again and using $f\circ\sigma=f$ yields
\[
g_1=\lambda(g_1\circ\sigma)+\mu(f\circ\sigma)
=\lambda(\lambda g_1+\mu f)+\mu f
=\lambda^2 g_1+(\lambda\mu+\mu)f.
\]
Hence $\lambda^2=1$ and $(\lambda+1)\mu=0$.
If $\lambda=-1$, then $g_1\circ\sigma=-g_1+\mu f$, and replacing $y$ by $\sigma(y)$ in
\eqref{eq:beta-neq-1-rewrite} (with $\beta=1$) gives
\begin{align*}
f(xy)&=f(x)g_1(\sigma(y))+g_1(x)f(\sigma(y))\\
&=f(x)\bigl(-g_1(y)+\mu f(y)\bigr)+g_1(x)f(y)\\
&=-f(x)g_1(y)+g_1(x)f(y)+\mu f(x)f(y).
\end{align*}
Let $\widehat g_1:=-g_1$. Then
\[
f(xy)=f(x)\widehat g_1(y)-\widehat g_1(x)f(y)+\mu f(x)f(y),
\]
contradicting Lemma~\ref{lem:b01}. Therefore $\lambda\neq -1$, so $\lambda=1$ and (since $\mathrm{char}(F)\neq 2$)
we must have $\mu=0$. Consequently,
\begin{equation}\label{eq:caseii-g1sigma-closed}
g_1\circ\sigma=g_1.
\end{equation}

Finally, apply Lemma~\ref{cor:AHLC-closure} to \eqref{eq:beta-neq-1-rewrite} with $\beta=1$, and use
\eqref{eq:caseii-beta1} and \eqref{eq:caseii-g1sigma-closed}. We obtain for all $x,y\in S$,
\[
f(xy)=f(x)g_1(\sigma(y))+g_1(x)f(\sigma(y))
=f(x)g_1(y)+g_1(x)f(y).
\]
Substituting $g_1=g+\frac{\gamma}{2}f$ gives
\[
f(xy)=f(x)g(y)+g(x)f(y)+\gamma f(x)f(y)\qquad(\forall x,y\in S).
\]
Moreover, from $g_1\circ\sigma=g_1$ and $f\circ\sigma=f$ we get
\[
g\circ\sigma+\frac{\gamma}{2}f=g_1\circ\sigma=g_1=g+\frac{\gamma}{2}f,
\]
hence $g\circ\sigma=g$. This proves \textup{(ii)}.
\end{proof}

\begin{corollary}\label{cor:AH44-central}
The conclusions of Theorem~\ref{thm:AH44} hold in particular if there exists a central element
$c\in Z(S)$ such that $f(c)\neq 0$.
\end{corollary}

\begin{proof}
By \eqref{eq:gen-anti}, taking $y=\sigma(c)$ shows that
\[
R(c)f=R(\sigma(\sigma(c)))f\in\mathrm{span}\{f,g\}.
\]
Since $c\in Z(S)$, we have $L(c)f=R(c)f\in\mathrm{span}\{f,g\}$.
Hence \eqref{eq:AH44-bridge} holds with $d=c$, and Theorem~\ref{thm:AH44} applies.
\end{proof}

\begin{remark}\label{rmk:role-centrality}
In the proof of Theorem~\ref{thm:AH44}, the dichotomy $\beta\in\{\pm1\}$ and the parity relation
$f\circ\sigma=\beta f$ are obtained before the bridge hypothesis \eqref{eq:AH44-bridge} is used. The bridge
assumption enters only through Lemma~\ref{lem:bridge-span}, which is invoked to show that $g\circ\sigma$ (or
$g_1\circ\sigma$ in the case $\beta\neq -1$) belongs to the span of the basic solution pair. This
span-closure step is then used to derive the standard $xy$-addition law and the explicit
$\sigma$-transformation rule for $g$. The central-element condition of
Corollary~\ref{cor:AH44-central} is merely a convenient sufficient condition for
\eqref{eq:AH44-bridge}; Example~\ref{ex:s3-bridge} in the sequel shows that it is strictly stronger.
\end{remark}

\section{Examples and Further Illustrations}\label{sec:examples}

We first illustrate the conjugation identity on familiar groups (hence semigroups).
In each of the first three examples, $\sigma$ is an involutive anti-automorphism and $J$ denotes composition with $\sigma$,
\[
(Jh)(x)=h(\sigma(x)).
\]
We verify the operator identity of Lemma~\ref{lem:conj}:
\[
J\,R(\sigma(y))\,J=L(y)\qquad(\forall\,y\in S).
\]
For the chosen finite-dimensional space $V$ of coordinate functions, we also describe explicitly how
the left-translation operators $L(y)$ act on $V$.

We emphasize that \textbf{Examples~1--3} below are included solely to illustrate the conjugation identity $J\,R(\sigma(y))\,J=L(y)$ and the resulting left-translation action (independent of the bridge hypothesis used in Theorem~\ref{thm:AH44}), whereas \textbf{Example~4} focuses on this bridge hypothesis to show that it may hold even when no central element $c$ satisfies $f(c)\neq0$.

\subsection{Example 1: \texorpdfstring{$\mathrm{GL}_n(K)$}{GL(n,K)} with transposition and column-coordinate functions}

Let $K$ be a field with $\operatorname{char}(K)\neq2$ and let $S=GL_n(K)$ with
$\sigma(A)=A^{\top}$. In this example, we take the coefficient field to be $F=K$.
Then $\sigma$ is an involutive anti-automorphism since $(AB)^{\top}=B^{\top}A^{\top}$ and $(A^{\top})^{\top}=A$.
Fix $v\in K^n$ with $v\neq 0$ and define, for $i=1,\dots,n$,
\[
f_i(A):=e_i^{\top}A\,v=(Av)_i,\qquad A\in GL_n(K),
\]
where $(e_1,\dots,e_n)$ is the standard basis of $K^n$. Put $V:=\mathrm{span}\{f_1,\dots,f_n\}$.

\medskip\noindent
\emph{Verification of $J R(\sigma(y)) J=L(y)$.}
Here $(Jf_i)(A)=f_i(A^{\top})=e_i^{\top}A^{\top}v$.
For $y\in GL_n(K)$ and $X\in GL_n(K)$,
\begin{align*}
\big(JR(\sigma(y))J f_i\big)(X)
&= \big(R(\sigma(y))Jf_{i}\big)(\sigma(X))
 = \big(R(y^{\top})Jf_{i}\big)(X^{\top}) \\
&= (Jf_{i})(X^{\top} y^{\top})
 = e_{i}^{\top} (X^{\top} y^{\top})^{\top} v
 = e_{i}^{\top} (y X) v \\
&= f_i(yX)
= (L(y)f_i)(X).
\end{align*}
Hence $J R(\sigma(y))J=L(y)$.

\medskip\noindent
\emph{Left action on $V$.}
For $y\in GL_n(K)$ and $X\in GL_n(K)$,
\[
(L(y)f_i)(X)=f_i(yX)=e_i^{\top}yXv=\sum_{j=1}^n y_{ij}\,e_j^{\top}Xv=\sum_{j=1}^n y_{ij}\,f_j(X).
\]
In particular, $L(y)V\subseteq V$ for every $y\in GL_n(K)$.

\subsection{Example 2: The symmetric group \texorpdfstring{$S_n$}{Sn}}

Let $S=S_n$ and $\sigma(\pi)=\pi^{-1}$.
Fix $j_0\in\{1,\dots,n\}$ and define, for $i=1,\dots,n$,
\[
f_i(\pi):=\mathbf{1}_{\{\pi(j_0)=i\}},\qquad \pi\in S_n,
\]
and set $V:=\mathrm{span}\{f_1,\dots,f_n\}$.

\medskip\noindent
\emph{Verification of $J R(\sigma(y)) J=L(y)$.}
Here $(Jf_i)(\pi)=f_i(\pi^{-1})=\mathbf{1}_{\{\pi(i)=j_0\}}$.
For $y\in S_n$ and $x\in S_n$,
\begin{align*}
\big(JR(\sigma(y))J f_i\big)(x)
&= \big(R(\sigma(y))Jf_{i}\big)(\sigma(x))
 = \big(R(y^{-1})Jf_{i}\big)(x^{-1}) \\
&= (Jf_i)(x^{-1}y^{-1})
 = (Jf_i)\big((yx)^{-1}\big)
 = f_i(yx)
= (L(y)f_i)(x).
\end{align*}
Hence $J R(\sigma(y))J=L(y)$.

\medskip\noindent
\emph{Left action on $V$.}
For $y\in S_n$ and $x\in S_n$,
\[
(L(y)f_i)(x)=\mathbf{1}_{\{(yx)(j_0)=i\}}
=\mathbf{1}_{\{x(j_0)=y^{-1}(i)\}}
=f_{y^{-1}(i)}(x).
\]
Thus $L(y)$ permutes the spanning set $\{f_1,\dots,f_n\}$, and in particular $L(y)V\subseteq V$.

\subsection{Example 3: \texorpdfstring{$\mathrm{SO}(3,\mathbb{R})$}{SO(3,R)} and column functions}

Let $S=SO(3,\mathbb{R})$ and $\sigma(g)=g^{-1}=g^{\top}$. Here the coefficient field is $F=\mathbb{R}$.
Write $g=(g_{i,\,j})_{1\le i,\,j\le 3}$ and define
\[
f_1(g):=g_{1,\,1},\qquad f_2(g):=g_{2,\,1},\qquad f_3(g):=g_{3,\,1},
\qquad V:=\mathrm{span}\{f_1,f_2,f_3\}.
\]

\medskip\noindent
\emph{Verification of $J R(\sigma(y)) J=L(y)$.}
We have $(Jf_i)(X)=f_i(X^{\top})=(X^{\top})_{i,\,1}=X_{1,\,i}$. For $y\in SO(3,\mathbb{R})$ and $x\in SO(3,\mathbb{R})$,
\begin{align*}
\big(JR(\sigma(y))J f_i\big)(x)
&= \big(R(\sigma(y))Jf_{i}\big)(\sigma(x))
= \big(R(y^{\top})Jf_{i}\big)(x^{\top}) \\
&= (Jf_{i})(x^{\top} y^{\top})
= f_i\big((x^{\top}y^{\top})^{\top}\big)
= f_i(yx)
= (L(y)f_i)(x).
\end{align*}
Hence $J R(\sigma(y))J=L(y)$.

\medskip\noindent
\emph{Left action on $V$.}
Since $(yx)_{i,\,1}=\sum_{j=1}^3 y_{i,\,j}\,x_{j,\,1}$, we get for $y,x\in SO(3,\mathbb{R})$,
\[
(L(y)f_i)(x)=f_i(yx)=(yx)_{i,\,1}=\sum_{j=1}^3 y_{i,\,j}\,f_j(x).
\]
In particular, $L(y)V\subseteq V$.

\medskip
\begin{remark}[Linear independence in Examples 1--3]\label{rmk:LI-examples}
The spanning families used above are linearly independent.

\smallskip\noindent
\textbf{Example 1.}
Assume $v\neq 0$.
If $\sum_{i=1}^n c_i f_i\equiv 0$, then for all $A\in GL_n(K)$,
\[
0=\sum_{i=1}^n c_i f_i(A)
=\Big(\sum_{i=1}^n c_i e_i^{\top}\Big)Av
=c^{\top}(Av).
\]
Since $GL_n(K)$ acts transitively on $K^n\setminus\{0\}$, we have $\{Av: A\in GL_n(K)\}=K^n\setminus\{0\}$,
hence $c=0$.

\smallskip\noindent
\textbf{Example 2.}
If $\sum_{i=1}^n c_i f_i\equiv 0$, choose $\pi$ with $\pi(j_0)=k$ to get $c_k=0$.

\smallskip\noindent
\textbf{Example 3.}
If $c_1 f_1+c_2 f_2+c_3 f_3\equiv 0$, evaluate at $I$ to get $c_1=0$.
Choose rotations sending $e_1$ to $e_2$ and $e_1$ to $e_3$ to get $c_2=c_3=0$.
\end{remark}

\subsection{Example 4: A non-central bridge instance on \texorpdfstring{$S_3$}{S3}}\label{ex:s3-bridge}

Let $S=S_3$ and $\sigma(\pi)=\pi^{-1}$. Define
\[
g(\pi):=\mathbf{1}_{A_3}(\pi),\qquad
f(\pi):=\mathbf{1}_{S_3\setminus A_3}(\pi),
\]
where $A_3$ denotes the alternating subgroup. Then $f$ and $g$ are linearly independent, and for all
$x,y\in S_3$ one has
\[
f(x\sigma(y))=f(xy^{-1})=f(x)g(y)+g(x)f(y).
\]
Indeed, $xy^{-1}$ is odd if and only if $x$ and $y$ have opposite parity. Thus \eqref{eq:gen-anti} holds
with $\beta=1$ and $\gamma=0$.

Choose any transposition $d\in S_3$. Then $f(d)=1$, and left multiplication by an odd permutation reverses parity, so
\[
(L(d)f)(x)=f(dx)=g(x)\qquad(\forall x\in S_3).
\]
Hence
\[
f(d)\neq0
\qquad\text{and}\qquad
L(d)f\in\mathrm{span}\{f,g\},
\]
so the bridge hypothesis \eqref{eq:AH44-bridge} holds.

On the other hand, $Z(S_3)=\{e\}$ and $f(e)=0$. Therefore no central element $c$ satisfies $f(c)\neq0$.
This shows that the bridge hypothesis of Theorem~\ref{thm:AH44} is strictly weaker than the
central-element condition in Corollary~\ref{cor:AH44-central}.

\medskip
The examples above show concretely how Lemma~\ref{lem:conj} converts the family $R(\sigma(y))$ into left translations $L(y)$
via conjugation by $J$, how this mechanism can be checked explicitly on natural finite-dimensional spaces of
coordinate functions, and why the intrinsic bridge hypothesis in Theorem~\ref{thm:AH44} is the right level of
generality for the structural argument.

\section{Concluding Remarks}\label{sec:conclusions}
The conjugation identity $J\,R(\sigma(y))\,J = L(y)$ provides an intrinsic operator interpretation of the
order-reversing family $R(\sigma(\cdot))$ that arises from an involutive anti-automorphism. Although an
equation of the form $f(x\sigma(y))=\cdots$ can also be rewritten by the substitution $y\mapsto\sigma(y)$,
the formula $J\,R(\sigma(y))\,J=L(y)$ shows directly how the anti-automorphic parameter family is converted
into a genuine left-translation action.
This left-translation viewpoint yields a natural Levi--Civita closure principle in the anti-automorphic
setting and leads, under the bridge hypothesis used in Theorem~\ref{thm:AH44}, to analogues of
Stetk\ae r's structural results for the generalized sine law, including the dichotomy
$\beta\in\{\pm1\}$, the standard $xy$-addition law, and the expected $\sigma$-transformation rules.
In the proof, the parity rule $f\circ\sigma=\beta f$ is obtained before the bridge step, whereas the bridge
hypothesis is used to control $g\circ\sigma$ and thereby to pass from the $x\sigma(y)$-law to the
standard $xy$-addition law.
Corollary~\ref{cor:AH44-central} shows that the existence of a central element $c$ with $f(c)\neq0$ is a
convenient sufficient condition for the bridge hypothesis, while Example~\ref{ex:s3-bridge} shows that this
sufficient condition is not necessary.
The examples illustrate the conjugation mechanism concretely on standard semigroups and show how the
relevant finite-dimensional spaces of coordinate functions behave under left translations.

\section*{Acknowledgment}

The author is deeply grateful to the handling editor for the careful handling of the manuscript and to the anonymous referee for reading the paper thoroughly and for providing valuable comments and suggestions, which led to a significant improvement in the presentation.

\subsection*{Funding.}
None.

\subsection*{Availability of data and materials}
Not applicable.

\subsection*{Declarations}
\subsection*{Conflict of interest}
None.

\bibliographystyle{plain}

\end{document}